\documentclass[english]{smfart}

\usepackage{smfthm}
\usepackage{enumerate}
\usepackage{amsmath}
\usepackage{amssymb}
\usepackage{amsfonts}

\newcommand{\mca}{\mathcal{A}}
\newcommand{\mco}{\mathcal{O}}

\newcommand{\rp}{\mathbb{R}}
\newcommand{\cp}{\mathbb{C}}

\newcommand{\pp}{\mathbb{P}}
\newcommand{\np}{\mathbb{N}}

\newcommand{\zp}{\mathbb{Z}}
\newcommand{\om}{\omega}
\newcommand{\ep}{\epsilon}

\newcommand{\wt}{\widetilde}

\newtheorem*{AbelTheorem}{Abel's Theorem}
\newtheorem*{AbelInverseTheorem}{Abel-inverse Theorem}
\newtheorem*{TheoremUn}{Theorem 2.2}
\newtheorem*{TheoremDeux}{Theorem 4.1}

\author{Martin Weimann}
\address{Universit\'e J. Fourier\\100 rue des Maths BP 74\\38402 St Martin d'Heres.}
\email{weimann23@gmail.com}

\title[Concavity, Abel-transform and the Abel-inverse theorem in toric...]{Concavity, Abel-transform and the Abel-inverse theorem in smooth complete toric varieties}

\begin{document}

\frontmatter

\begin{abstract}
We extend the usual projective Abel-Radon transform to the larger context of a smooth complete toric variety $X$. We define and study toric $E$-concavity attached to a split vector bundle on $X$. Then we obtain a multidimensional residual representation of the toric Abel-transform and we prove a toric version of the classical Abel-inverse theorem.
\end{abstract}

\begin{altabstract}
Nous \'etendons la transform\'ee d'Abel-Radon projective au cadre plus large d'une vari\'et\'e torique lisse compl\`ete $X$. Pour ce faire, nous d\'efinissons et \'etudions dans un premier temps la notion de $E$-concavit\'e torique attach\'ee \`a un fibr\'e vectoriel scind\'e $E$ sur $X$. Finalement, nous d\'efinissons la transform\'ee d'Abel torique et nous prouvons une version torique du th\'eor\`eme d'Abel-inverse \`a l'aide du calcul r\'esiduel multivari\'e.
\end{altabstract}

\maketitle
\tableofcontents

\mainmatter

\section{Introduction}
Abel-tranform's philosophy was introduced by Abel in its pionneer article on integrals of rational forms on algebraic curves. He considered sums of such integrals depending on $0$-cycles given by the intersection of the given curve $V$ with a rational family of algebraic curves $(C_a)_{a\in A}$ and made the fundamental discovery that such sums can be expressed in terms of rational and logarithmic functions of the parameter $a$. 

If we suppose that $V$ is a curve in $\pp^2$ and $C_a$ are lines of $\pp^2$ (so that $A=(\pp^{2})^*$), then Abel's theorem can be rephrased as follow. Consider the incidence variety $I_V=\{(x,a)\in V\times (\pp^{2})^*, x\in C_a\}$ over $V$ and its associated diagram
\begin{eqnarray*}
& I_V & \\
p &\swarrow \quad \searrow & q \\
V & \quad \quad & (\pp^2)^*
\end{eqnarray*}
where $p$ and $q$ are the natural projections. 
{\AbelTheorem
For any rational $1$-form $\phi$ on $V$, the current $q_* p^* \phi$ coincides with a rational $1$-form on $(\pp^{2})^*$.}
\vskip2mm
\noindent
The original proof of Abel deals with more general curves than lines. We state it for lines for simplicity, since the proof is immediate in both cases with the modern language of currents. 

The map $T\mapsto q_* p^* T$ on currents is well-defined since $p$ is a submersion and $q$ is proper. This is the so called Abel-transform. On the Zariski open set of $(\pp^{2})^*$ consisting of elements $a$ for which the intersection $V\cdot C_a$ is transversal and does not meet the polar locus of $\phi$, the current $q_* p^* \phi$ is holomorphic, equal to the trace form
$$
Tr_V\phi (a):=\sum_{p\in V\cap C_a} \phi(p).
$$
Abel's theorem gave new perspectives as well in complex analysis as in algebraic geometry and number theory. In particular, studying analytic sets in the complex projective space $X=\mathbb{P}^n$ using scanning by linear subspaces of complementary dimension gave rise to an intense activity around inversion problems by several mathematiciens, as Lie, St-Donnat \cite{SD:gnus}, Wirtinger, Darboux, Griffiths \cite{Griffiths:gnus}, Henkin \cite{henkin1:gnus, henkin:gnus}, Passare \cite{hp:gnus}. Let us explain the so-called Abel-inverse Theorem obtained by Saint-Donat (\cite{SD:gnus}, 1975) and Henkin (\cite{henkin1:gnus}, 1992) for $n=2$.

We consider $U\subset \pp^2$ an open neighborhood of a line $C_{\alpha}\subset \pp^2$ and $V\subset U$ an analytic curve transversal to $C_{\alpha}$. For $U$ small enough, we can define as before the trace of any holomorphic $1$-form $\phi$ on $V$, which is a holomorphic $1$-form $Tr_V\phi$ on the open set $U^{'}\subset (\pp^{2})^*$ of lines included in $U$. 

{\AbelInverseTheorem 
Let $N$ be the cardinality of $V\cap C_{\alpha}$. The analytic curve $V$ is contained in an algebraic curve  $\wt{V}\subset\mathbb{P}^2$ of degree $N$ and $\phi$ extends to a rational (resp. regular) form $\wt{\phi}$ on $\wt{V}$ if and only if the form $Tr_V \phi$ is rational (resp. vanishes) on $(\pp^{2})^*$.}
\vskip2mm
\noindent
Let us insist on the important contributions of Griffiths (\cite{Griffiths:gnus}, 1976) who introduced Grothendieck residues in the picture and of Henkin (\cite{henkin:gnus}, 1995) who introduced the notion of Abel-Radon transform. Latter on, Henkin and Passare (\cite{hp:gnus}, 1999) established the link with multidimensional residue theory, in particular the modern formalism of residue currents, and proved in \cite{hp:gnus} a stronger local version of the Abel-inverse theorem which holds for any dimension and for analytic subset $V\subset U\subset \pp^n$ of any pure dimension $k$ in an open neighborhood $U$ of a $(n-k)$-linear subspace, and for $\phi$ a meromorphic $k$-form on $V$. Finally, Fabre generalized in his thesis \cite{Fabre:gnus} the Abel-inverse theorem to the more general situation of complete intersections with fixed multi-degree replacing linear subspaces.

Our main goal here is to give a toric version of the Abel-inverse theorem, where a smooth complete toric variety $X$ replaces $\pp^n$, and zero locus of sections of a globally generated split vector bundle $E=L_1\oplus\cdots\oplus L_{k}$ on $X$ replaces linear subspaces. This article has been motivated by the previously mentionned result of Fabre and by results of \cite{WM2:gnus}, where the author relates the inversion mechanisms to a rigidity propriety of a certain system of PDE's to give finally a stronger form of the classical Abel-inverse theorem.

When $n\ge 3$, most of complete toric varieties are not projective, so that we can not {\it a priori} deduce the toric Abel-inverse theorem (theorem~\ref{t3.2}) from the classical one \cite{hp:gnus} when $X=\pp^n$. Thus, our proof does not rely on a projective embedding of $X$ and follows the one given in \cite{WM2:gnus}. Moreover, recent results of \cite{and:gnus, shchuplev:gnus, STY:gnus} suggest that there are canonical kernels for residue currents in complete toric varieties $X$ who replaces the classical Cauchy or Bochner-Martinelli kernels. This is an other important motivation for an intrinsec approach of Abel-tranform in toric varieties (even projective ones), especially in view of effectivity aspects in Abel-inverse problems. 

Let us describe the content of the article.
\vskip2mm
\noindent
{\it Section~\ref{sec:2}.} We define and study the notion of toric concavity attached to the bundle $E$ on the toric variety $X$. In analogy with the projective case, an open set $U$ of $X$ (for the usual topology) is said to be $E$-concave if any $x$ in $U$ belongs to a subscheme $C=\{s=0\}$, $s\in\Gamma(X,E)$ included in $U$. Contrarly to the projective case, the toric context brings us to study which bundles $E$ give rise to families of subvarieties which can move ``sufficiently'' in $X$, that his bundles $E$ for which there exist non trivial $E$-concave open sets in $X$. 

The main theorem of that section concerns the set theoretical orbital decomposition of such subschemes, called  $E$-subscheme. 

{\TheoremUn
A generical $E$-subscheme can be uniquely decomposed as the cycle
$$
C=\sum \nu_{I,\tau}C_{I,\tau},\quad \nu_{I,\tau}\in \{0,1\}
$$
where $\tau$ runs over the cones of the fan $\Sigma$ of $X$, $I$ runs over the subsets of $\{1,\ldots,k\}$ and the $C_{I,\tau}$ are smooth $|I|$-codimensional subvarieties of the toric subvariety $V(\tau)\subset X$ attached to $\tau$, with transversal or empty intersection with all orbits closures included in $V(\tau)$.}

The integers $\nu_{I,\tau}\in \{0,1\}$ are explicitly determined from geometry of polytopes. If $E$ is globally generated, they are all zero except possibly for $I=\{1,\ldots,k\}$ and $\tau=\{0\}$ (so that $V(\tau)=X$). We state precisely and prove theorem~\ref{t2.1} in section~\ref{sec:2.2}, after some reminders on toric geometry in section~\ref{sec:2.1}.

We introduce $E$-concavity in section~\ref{sec:2.3}. Let 
$$
X^{'}=X^{'}(E):=\pp(\Gamma(X,L_1))\times\cdots\times \pp(\Gamma(X,L_k))
$$
be the parameter space corresponding to $E=L_1\oplus\cdots\oplus L_k$. Any $a\in X^{'}$ naturaly defines a closed $E$-subscheme $C_a\subset X$. We deduce from theorem~\ref{t2.1} that the $E$-dual set 
$$
U^{'}:=\{a\in X^{'}, C_a\subset U\}
$$
of an $E$-concave open set $U$ is open in $X^{'}$ if and only if the bundle $E$ is globally generated, and the family $(L_1,\ldots,L_k)$ satisfies an additional combinatorial condition on the involved polytopes, called condition of essentiality. We show that these conditions are equivalent to that the projections $p_U$ and $q_U$ of the incidence variety 
$$
I_U=\{(x,a)\in X\times X^{'}, x\in C_a\}
$$
on $U$ and $U^{'}$ are respectively a submersion and a submersion over a Zariski open set ${\rm Reg}\,U^{'}\subset U^{'}$. This is the reasonnable situation to generalize the usual Abel-transform, as explained in \cite{Fabre:gnus}.

We study in section~\ref{sec:2.4} the $E$-dual set 
$$
V^{'}:=\{a\in U^{'}, C_a\cap V\ne \emptyset\}
$$
of an analytic subset $V$ of an open $E$-concave subset $U$ of $X$. Contrarly to the projective space where we profit that the Picard group is simply $\zp$, there are in the toric context pathologic situations of analytic subsets whose intersection with $E$-subschemes is never proper. We characterize those degenerated subsets in the algebraic case $U=X$. Finally, we extend $E$-duality to the case of cycles and, using resultant theory, we compute the multidegree (in the product of projective space $X^{'}$) of the divisor $E$-dual to a $(k-1)$-cycle of $X$.

\begin{rema}
Linear concavity and convexity play an important role in complex analysis when studying varying transforms as Abel, Radon, Abel-Radon or Fantappi\'e-Martineau transforms (see \cite{AndPasSig:gnus, henkin:gnus} for instance). In that spirit, recent results \cite{shchuplev:gnus, STY:gnus} suggest that concavity in smooth complete toric varieties as developed here should simplify the description of various transforms, using some global intrinsec integral representation for residue currents \cite{and:gnus}, whose kernel is canonically associated to $X$ and $E$.
\end{rema}
\noindent
{\it Section~\ref{sec:3}.} We generalize the Abel-transform to the toric context and we prove the toric version of the Abel-inverse theorem, using the Cauchy residual representation of the Abel-transform. 

After reminders in section~\ref{sec:3.1} about meromorphic and regular forms on analytic sets, we define in section~\ref{sec:3.2} the toric Abel-transform attached to an essential globally generated vector bundle $E=L_1\oplus\cdots\oplus L_k$. For any closed analytic subset $V$ of an $E$-concave set $U\subset X$, we can multiply the current of integration $[V]$ with any meromorphic $q$-form $\phi$ on $V$. The Abel-transform of the resulting current $[V]\land \phi$ is the current on $U^{'}$ defined by
$$
\mathcal{A}([V]\land\phi):=q_{U*}p_{U}^*([V]\land \phi)
$$
where $p_{U}^*([V]\land \phi):=[p_U^{-1}(V)]\land p_{U}^*\phi$. If $V$ is smooth and $k$-dimensional with transversal intersection 
$$
V\cap C_a=\{p_1(a),\ldots,p_N(a)\}
$$
for $a$ in $U^{'}$, and if $\phi$ is a global section of the sheaf $\Omega^q_V$ of holomorphic $q$-forms on $V$, the current $\mathcal{A}([V]\land\phi)$ is a $\bar{\partial}$-closed $(q,0)$-current on $U^{'}$ which coincides with the $q$-holomorphic form
$$
Tr_V\phi(a):=\sum_{j=1}^N p_j^*(\phi)(a),
$$
called the trace of $\phi$ on $V$ according to $E$. If $V$ is singular and $\phi$ is meromorphic, the current representation of the trace implies easily, as in the projective case \cite{hp:gnus}, that the $q$-form $Tr_V\phi$ originally defined and holomorphic on a dense open subset of $U^{'}$ extends to a $q$-meromorphic form on $U^{'}$. Moreover, the map $\phi\mapsto Tr_V\phi$ commutes with the operators $d$, $\partial$ and $\bar{\partial}$ and thus induces a morphism on the corresponding graded complex vector spaces.

In section~\ref{sec:3.3}, we give, as in \cite{WM2:gnus}, a residual representation of the form $Tr_V\phi$, using Grothendieck residues and Cauchy integrals depending meromorphically of the parameter $a$. As mentioned before, it should be interesting to obtain an intrinsec integral representation of the current $Tr_V\phi$ using a global kernel constructed from the toric variety $X$ and the bundle $E$ (see \cite{and:gnus, shchuplev:gnus} for such motivations).

Finally, we show in section~\ref{sec:3.4} the toric version of the Abel-inverse theorem in the hypersurface case, under its algebraic strongest form stated in \cite{WM2:gnus}. We assume that $E=L_1\oplus\cdots\oplus L_{n-1}$ is an essential globally generated bundle which satisfies the following additional condition~: there exists an affine  chart $U_{\sigma}\simeq \cp^n$ of $X$ associated to a maximal cone $\sigma$ of the fan $\Sigma$, such that for any $n-1$-dimensional cone $\tau\subset\sigma$, one has $dim\, H^0(V(\tau),L_{i|V(\tau)})\ge 2$ for $i=1,\ldots,n-1$, where $V(\tau)$ is the one dimensional toric subvariety of $X$ associated to $\tau$. There always exists such a bundle $E$ on $X$, even if $X$ is not projective.

{\TheoremDeux
Let $V\subset U$ be a codimension 1 analytic subset of a connected $E$-concave open set $U\subset X$ with no components in the hypersurface at infinity $X\setminus U_{\sigma}$. Let $\phi$ be a meromorphic $(n-1)$-form on $V$ not identically zero on any component of $V$. Then there exists an hypersurface $\wt{V}\subset X$ such that $\wt{V}_{|U}=V$ and a rational form $\wt{\phi}$ on $\wt{V}$ such that $\wt{\phi}_{|V}=\phi$ if and only if the meromorphic form $Tr_V\phi$ is rational in the constant coefficients $a_0=(a_{1,0},\ldots,a_{n-1,0})$ of the $n-1$ polynomial equations of $C_a$ in the affine chart $U_0$.}

Let us mention that this theorem is intimely linked to the toric interpolation result in \cite{WM1:gnus}, where $\phi$ is replaced by some rational function on $X$.

\section{Toric concavity}\label{sec:2}

\subsection{Preleminaries on toric geometry}\label{sec:2.1}

We refer to \cite{Danilov:gnus}, \cite{Ewald:gnus} and \cite{Fulton:gnus} for basic references on toric geometry.

\subsubsection{Basic notions}\label{sec:2.1.1}

Let $X$ be an $n$-dimensional smooth complete toric variety associated to a complete regular fan $\Sigma$ in $\rp^n$. We note $\mathbb{T}:=(\cp^*)^n$ the algebraic torus contained in $X$ equipped with its canonical coordinates $t=(t_1\ldots,t_n)$. For $m=(m_1,\ldots,m_n)\in \zp^n$,  we denote by $t^m$ the Laurent monomial $t_1^{m_1}\cdots t_n^{m_n}$. We note $\mco_X$ the structure sheaf and $\cp(X)$ the field of rational functions on $X$.

The set of irreducible $k$-codimensional subvarieties of $X$ invariant under the torus action (called $\mathbb T$-subvarieties) is in one-to-one correspondence with the set $\Sigma(k)$ of cones of dimension $k$ in $\Sigma$. For a cone $\tau$ in $\Sigma$, we note $V(\tau)$ the corresponding subvariety, $\tau(r)$ the set of $r$-dimensional cones of $\Sigma$ included in $\tau$ and 
$$
\check{\tau}:=\{m\in\zp^n, \langle m,\eta_{\rho} \rangle \ge 0\quad \forall \rho \in \tau(1)\}
$$
the dual cone of $\tau$, where $\eta_{\rho}\in\zp^n$ is the unique primitive vector of the monoid $\rho\cap \zp^n$ and $\langle \cdot,\cdot \rangle$ denotes the usual scalar product in $\rp^n$. If $\sigma$ is a cone of maximal dimension $n$, the corresponding affine toric variety $U_{\sigma}:={\rm Spec} \cp[\check{\sigma}\cap \zp^n]$ is isomorphic to $\cp^n$. The compact toric variety $X$ is obtained by gluing together the affine charts $U_{\sigma}$ and $U_{\sigma'}$ along their common open set $U_{\sigma\cap \sigma'}$. The associated transition maps are given by monomial applications induced by the change of basis from $\check{\sigma}\cap \zp^n$ to $\check{\sigma'}\cap \zp^n$. If the monoid $\check{\sigma}\cap \zp^n$ admits $(m_1(\sigma),\ldots,m_n(\sigma))$ as a $\zp$-basis, we note
$$
x_{\sigma}=(x_{1,\sigma},\ldots,x_{n,\sigma}):=(t^{m_1(\sigma)},\ldots,t^{m_n(\sigma)})
$$
the associated {\it canonical system of affine coordinates} in the chart $U_{\sigma}$ corresponding bijectively to the one-dimensional cones $\rho_1,\ldots,\rho_n$ included in $\sigma$. We note 
$$
\phi_{\sigma}:\rp^n\to \rp^n
$$
the isomorphism which sends $m_i(\sigma)$ on the $i$-th vector $e_i$ of the canonical basis of $\rp^n$ such that $e_i$ corresponds to $x_{i,\sigma}$. 
If $\tau\in \sigma(k)$ is generated by $\{\rho_1,\ldots,\rho_k\}$, the toric variety $V(\tau)$ intersected with the affine chart $U_{\sigma}$ corresponds to the coordinate linear subspace
$$
V(\tau)_{|U_{\sigma}}=\{x_{1,\sigma}=\cdots=x_{k,\sigma}=0\}.
$$
The variety $X$ has the set theoretical representation  
$$
X=\mathbb{T}\cup\bigcup_{\rho\in\Sigma(1)} D_{\rho},
$$
where the $D_{\rho}$'s are the irreducible $\mathbb T$-divisors attached to the one-dimensional cones (the rays) of $\Sigma$. Any $\mathbb T$-divisor $D$ on $X$,
$$
D=\sum_{\rho\in\Sigma(1)} k_{\rho} D_{\rho},\quad k_{\rho}\in \zp,
$$
is a Cartier divisor with local data $(U_{\sigma}, t^{s_{\sigma,D}})_{\sigma\in \Sigma(n)}$ where the vectors $s_{\sigma,D}\in \zp^n$ are uniquely determined by the $n$ equalities
$$
\langle s_{\sigma,D},\eta_{\rho} \rangle= k_{\rho} \quad\forall \rho\in \sigma(1).
$$
It is well known \cite{Fulton:gnus} that the vector space $\Gamma(X,L(D))$ of global sections of the line bundle $L(D)$ attached to $D$ is isomorphic to the vector space of Laurent polynomials supported in the convex integer polytope 
$$
P_D=\{m\in \rp^n, \langle m,\eta_{\rho} \rangle + k_{\rho} \ge 0 \quad\forall \rho\in \sigma(1)\}.
$$
Such a Laurent polynomial $f=\sum_{m\in P_D\cap \zp^n}c_mt^m$ defines a global section of $L(D)$ given in the affine chart $U_{\sigma}$ by the polynomial 
$$
f^{\sigma}\in \cp[x_{1,\sigma},\ldots,x_{n,\sigma}]
$$
expressing the Laurent polynomial $t^{-s_{\sigma,D}}f$ (which belongs to $\cp [\check{\sigma}\cap \zp^n]$ by assumption) in the affine coordinates $x_{\sigma}$. The polynomial $f^{\sigma}$ is supported in the polytope $\Delta_{D,\sigma}$ (contained in $(\rp^+)^n$), image of the translated polytope $P_D-s_{\sigma,D}$ under the isomorphism $\phi_{\sigma}$ defined before.

The complete linear system $|L(D)|$ of effective divisors rationally equivalent to $D$ is isomorphic to the projective space $\pp(\Gamma(X,L(D)))\simeq \pp^{l(D)-1}$ where $l(D)$ is the cardinality of $P_D\cap \zp^n$.

\subsubsection{Globally generated line bundles}\label{sec:2.1.2}

We refer to \cite{WM3:gnus} for this part. Any effective $\mathbb{T}$-divisor $D=\sum_{\rho\in\Sigma(1)} k_{\rho} D_{\rho}$ on $X$ can be uniquely decomposed in a sum $D=D'+D''$ of a ``mobile'' divisor $D'$ corresponding to a globally generated line bundle and a ``fixed'' effective divisor $D''$. That means that the zero divisor of any section $s\in\Gamma(X,L(D))$ satisfies
$$
{\rm div}_0 s = {\rm div}_0 s' + D'',
$$
where $s'\in \Gamma(X,L(D'))$. The mobile divisor $D'$ is equal to 
$$
D'=\sum_{\rho\in\Sigma(1)} k_{\rho}' D_{\rho}
$$
where $k_{\rho}':=-\min_{m\in P_E} \langle m,\eta_{\rho} \rangle$. Note that $0\le k_{\rho}'\le k_{\rho}$ so that $D''=D-D'$ is effective. 

For any cone $\tau\in \Sigma$, we consider the two convex polytopes 
\begin{eqnarray*}
P_D^{\tau}&:=&\{m\in P_E, \langle m,\eta_{\rho} \rangle =-k'_{\rho}\quad \forall \rho \in \tau(1)\}\quad{\rm and}\\
P_D^{(\tau)}&:=&\{m\in P_E, \langle m,\eta_{\rho} \rangle =-k_{\rho}\quad \forall \rho \in \tau(1)\}.
\end{eqnarray*}
We call $P_D^{\tau}$ the face of $P_D$ associated to $\tau$ and $P_D^{(\tau)}$ the virtual face of $P_D$ associated to $\tau$, since it can be empty \cite{Ewald:gnus}. The $\mathbb{T}$-subvariety $V(\tau)$ is included in the base locus
$$
BS(L(D)):=\{x\in X, s(x)=0 \quad \forall s \in \Gamma(X,L(D))\}
$$
of the linear system $|L(D)|$ if and only if the associated virtual face $P_D^{(\tau)}$ is empty. In particular, considering $\tau=\sigma\in\Sigma(n)$ (for which $V(\sigma)$ is the fixed point corresponding to the origin of the chart $U_{\sigma}$), we remark that the line bundle $L(D)$ is globally generated if and only if the vectors $s_{\sigma,D}$ belong to $P_D$ (that is $0\in \Delta_{D,\sigma}$) for all $\sigma\in\Sigma(n)$. That means that in any affine chart, the polynomial equation of a generic divisor $H\in |L(D)|$ has a non zero constant term.

If the line bundle $L(D)$ is globally generated, it restricts for any $\tau\in \Sigma$ to a globally generated line bundle $L(D)^{\tau}$ on $V(\tau)$, whose polytope can be naturally identified with the polytope $P_D^{\tau}$, equal to $P_D^{(\tau)}$ in that case.

\subsection{The orbital decomposition theorem}\label{sec:2.2}

Let $L_1,\ldots  ,L_k$ be a family of line bundles on $X$ attached to a collection of effective $\mathbb T$-divisors $D_1,\ldots  ,D_k$, with polytopes $P_{1}=P_{D_1},\ldots,P_{k}=P_{D_k}$ defined as before. We note $E=L_1\oplus\cdots\oplus L_k$ the associated rank $k$ vector bundle. We say that a subscheme $C\subset X$ is an $E$-subscheme if it is the zero set of a global section of $E$. 

This section deals with the generic structure of an $E$-subschemes, where generically means for $s$ in a Zariski open set of the vector space $\Gamma(X,E)$. 

We use the following definition introduced in \cite{Kho:gnus} for $k=n$:
\begin{defi} 
A family $(P_1,\ldots  ,P_r)$ of polytopes in $\rp^n$ is called essential if for any subset $I$ of $\{1,\ldots  ,r\}$, the dimension of the Minkowski sum $\sum_{i\in I}P_i$ is at least the cardinality $|I|$ of $I$. The bundle $E$ (or the family $L_1,\ldots,L_k$) is called essential if the associated family of polytopes is.
\end{defi}

Let us state the orbital decomposition theorem, using notations from section~\ref{sec:2.1}.

\begin{theo}\label{t2.1}
A generic $E$-subscheme can be uniquely decomposed as the cycle
$$
C=\sum_{\begin{matrix} I\subset \{1,\ldots  ,k\}\\ \tau\in \Sigma
\end{matrix}} \nu_{I,\tau}C_{I,\tau}
$$
where the $C_{I,\tau}$ are smooth subvarieties, complete intersection of codimension $|I|$ in $V(\tau)$, with transversal or empty intersection with orbits included in $V(\tau)$, and the integers $\nu_{I,\tau}\in \{0,1\}$
are defined by:
$$
\nu_{I,\tau}=\begin{cases} \displaystyle 1 \,\,{\rm \textit{if}}\,\begin{cases}  \forall i\notin I, P_{i}^{(\tau)}=\emptyset\\
                                                       \forall \tau'\subset \tau, \exists i\notin I, P_{i}^{(\tau')}\ne\emptyset\quad and\\
{\rm \textit{the\, family}}\,(P_{i}^{\tau})_{i\in I} \,{\rm \textit{is\, essential}}
                                                          \end{cases}\\
0 \,\, {\rm \textit{otherwise}}.
\end{cases}
$$
\end{theo}

The main ingredient of the proof is the following proposition and its corollary, consequence of Sard's theorem:

\begin{prop}\label{p2.1} If $E$ is globally generated, a generic $E$-subscheme is a smooth complete intersection if and only if $E$ is essential, and is empty otherwise.
\end{prop} 

\begin{proof} For $s=(s_1,\ldots,s_k)\in \Gamma(X,E)$, and  $\sigma\in \Sigma(n)$, we have 
$$
\{s=0\}\cap U_{\sigma}=\{f_1^{\sigma}=\cdots =f_k^{\sigma}=0\}
$$
where the polynomials $f_i^{\sigma}\in\cp[x_1^{\sigma},\ldots  ,x_n^{\sigma}]$ are supported by the convex polytopes $\Delta_{D_i,\sigma}$ associated to the divisor $D_i$. Clearly, the family $(P_{1},\ldots  ,P_{k})$ is essential if and only if the family $(\Delta_{D_1,\sigma},\ldots  ,\Delta_{D_k,\sigma})$ is. 
This is equivalent to that each polytope $\Delta_{D_i,\sigma}$ contains a vector $e_i$ of the canonical basis of $\rp^n$ such that the family $(e_1,\ldots  ,e_k)$ is free. This implies that the differential form  $df_1^\sigma \wedge \cdots \wedge df_k^\sigma$ is generically non zero, since it contains a non zero summand of the form $g^{\sigma}dx_1^{\sigma}\land \cdots\land dx_k^{\sigma}$, $g^{\sigma}\ne 0$. Since the $L_i$ are globally generated, the polytopes $\Delta_{D_i,\sigma}$ contain the origin. Thus, the Zariski closure in $X$ of the affine set define by the polynomials $f_1^\sigma -\epsilon_1,\ldots  , f_k^\sigma -\epsilon_k$ remains an $E$-subscheme which is, by Sard's theorem, a smooth complete intersection in $X$ for generic $\epsilon_i$, $i=1,\ldots,k$. If $E$ is not essential, there is a subset $I=\{i_1,\ldots,i_r\}$ of $\{1,\ldots,k\}$ for which the family $\{f_{i_1}^{\sigma},\ldots,f_{i_r}^{\sigma}\}$ of $r$ polynomials depends on strictly less than $r$ variables in any chart $U_{\sigma}$, so that $C=\{s=0\}$ is generically empty.
\end{proof} 

\begin{rema} 
In the essential globally generated case, $E$-subschemes are not necessary generically irreducible, but are disjoint union of irreducible components. For example, if $E=\mathcal{O}_{\pp^1\times \pp^1}(2,0)$, a generic $E$-subscheme consists of two disjoint lines $\{p_1\}\times \pp^1$ and $\{p_2\}\times \pp^1$.
\end{rema}

\begin{coro}\label{c2.1}
Suppose that $E$ is globally generated. For any $\tau\in \Sigma(r)$ and any generic $E$-subscheme $C$, the intersection $C\cap V(\tau)$ is transversal and define 
a smooth subvariety $C_{\tau}$ of $X$ of codimension $r+k$ if and only if the family  $P_{1}^{\tau},\ldots  ,P_{k}^{\tau}$ is essential, and $C\cap V(\tau)$ is empty otherwise.
\end{coro}

\begin{proof} The line bundles $L_1,\ldots,L_k$ restricted to $V(\tau)$ define $k$ line bundles $L_1^{\tau},\ldots,L_k^{\tau}$, globally generated on the toric variety $V(\tau)$, with associated polytopes $P_{1}^{\tau},\ldots,P_{k}^{\tau}$. We then apply proposition~\ref{p2.1}.
\end{proof}

\begin{proof}[Proof of Theorem~\ref{t2.1}]  Any $E$-subscheme $C=\{s_1=\cdots=s_k=0\}$, $s_i\in \Gamma(X,L_i)$, $i=1,\ldots,k$ can be written as
\begin{eqnarray*}
C=\bigcap_{i\in \{1,\ldots,k\}}\Big(\{s'_i=0\}\cup BS(L_i)\Big)=\bigcup_{I\subset \{1,\ldots,k\}}\Big(\bigcap_{i\in I}\{s'_i=0\}\bigcap_{i\notin I} {\rm BS}\, (L_i)\Big)
\end{eqnarray*}
where, as in section~\ref{sec:2.2}, $\{s'_i=0\}$ is the mobile part of the divisor $\{s_i=0\}$. A $\mathbb{T}$-subvariety $V(\tau)$ is included in the intersection of base locis $\bigcap_{i\notin I} {\rm BS}\, (L_i)$ if and only if the virtual faces $P_i^{(\tau)}$ are empty for all $i\notin I$. Moreover, there are no bigger $\mathbb{T}$-subvarieties containing such a $V(\tau)$ and included in $\bigcap_{i\notin I} {\rm BS}\, (L_i)$ if and only if for all $\tau'\subset \tau$, there exists $i\notin I$ such that the virtual face $P_i^{(\tau')}$ is not empty. 

Since the $s'_i$ are global sections of a globally generated line bundle $L(D'_i)$ ($D'_i$ is the mobile part of the divisor $D_i$), corollary~\ref{c2.1} implies that 
$$
C_{I,\tau}\,:=\,V(\tau)\,\cap \,\bigcap_{i\in I}\,\{s'_i=0\}
$$ 
is generically a smooth complete intersection of codimension $|I|$ in $V(\tau)$ if the family $(P_i^{\tau})_{i\in I}$ is essential, empty otherwise. Again by corollary~\ref{c2.1}, $C_{I,\tau}$ has a transversal or empty intersection with all orbits included in $V(\tau)$.
\end{proof}

Let us introduce Bernstein's theorem \cite{Ber:gnus} in the picture:

\begin{prop}\label{p2.2}
If $E$ is globally generated and $dim \,V(\tau)= rank\, E=k$, the intersection number $V(\tau)\cdot C$ for a generic $E$-subvariety $C$ is positive, equal to the $k$-dimensional mixed volume
$$
V(\tau)\cdot C=MV_{k}(P_{1}^{\tau},\ldots  ,P_{k}^{\tau}).
$$
\end{prop}

\begin{proof}
Choose $\sigma\in \Sigma(n)$ containing $\tau$, such that 
$$
V(\tau)\cap U_{\sigma}=\{x_{k+1,\sigma}=\cdots =x_{n,\sigma}=0\}
$$
and let $f_1^{\sigma}=\cdots =f_k^{\sigma}=0$ be affine equations of $C\cap U_{\sigma}$. Since  the line bundles $L_i^{\tau}$ are globally generated on the toric variety $V(\tau)$, corollary~\ref{c2.1} implies that the intersection $C\cap V(\tau)$ is generically finite, transversal to the orbits included in $V(\tau)$. It is thus included in $V(\tau)\cap U_{\sigma}$ and globally given by polynomials equations
$$
C\cap V(\tau)=\{f_1^{\sigma,\tau}(x_{1}^{\sigma},\ldots  ,x_k^{\sigma})=\cdots 
=f_k^{\sigma,\tau}(x_{1}^{\sigma},\ldots  ,x_k^{\sigma})=0\},
$$
where $f_i^{\sigma,\tau}(x_{1}^{\sigma},\ldots  ,x_k^{\sigma}):=f_i^{\sigma}(x_{1}^{\sigma},\ldots  ,x_k^{\sigma},0,\ldots  ,0)$. The support $\Delta_{i,\sigma}^{\tau}$ of a generic polynomial $f_i^{\sigma,\tau}$ is the subset of $(\rp^+)^{k}\times 0_{(\rp^+)^{n-k}}$, image of the polytope $P_{i}^{\tau}$ under the isomorphism $\phi_{\sigma}:\rp^n\to\rp^n$ and has the same normalized volume. From Bernstein's theorem, we have
$$
{\rm Card} (C\cap V(\tau))={\rm MV}_{k}(\Delta_{1,\sigma}^{\tau},\ldots,\Delta_{k,\sigma}^{\tau})={\rm MV}_{k}(P_{1}^{\tau},\ldots  ,P_{k}^{\tau}).
$$
Since the intersection $C\cap V(\tau)$ is supposed to be transversal, the intersection number $C\cdot V(\tau)$ is equal to ${\rm Card} (C\cap V(\tau))$. 
\end{proof}

Since the classes of $k$-dimensional $\mathbb T$-subvarieties generate the $k$-Chow group $A_{k}(X)$ of algebraic $k$-cycles of $X$ modulo rational equivalence \cite{Fulton:gnus}, the previous proposition permits to compute the intersection number $V\cdot C$ of any $k$-dimensional closed subvariety $V$ of $X$ with a generic $E$-subscheme $C$. 

\begin{coro}\label{c2.2} 
The stricly positivity conditions $V(\tau)\cdot L_1\cdots L_k>0$ are satisfied for any $\tau\in \Sigma(n-k)$ if and only if the bundle $E$ is globally generated, essential, and if the line bundle $L_1\otimes\cdots\otimes L_k$ is very ample.
\end{coro}

\begin{proof}
The $k$-dimensional mixed volume of a family of $k$ polytopes in $\rp^k$ is strictly positive if and only if the family is essential (see \cite{Ewald:gnus}). Essentiality of the families $P_{1}^{\tau},\ldots  ,P_{k}^{\tau}$ for all $\tau\in \Sigma(n-k)$ is equivalent to that any face $P^{\tau}$ of the Minkowski sum $P:=P_1+\cdots+P_k$ has maximal dimension $k$. In the globally generated case, the polytope $P=P_D$ corresponds to the divisor $D=D_1+\cdots+D_k$. It satisfies the previous condition if and only if the translated polytope $P_D-s_{\sigma,D}$ contains the basis of the $\zp$-free module $\check{\sigma}\cap \zp^n$ for all $\sigma\in \Sigma(n)$. This holds if and only if the associated line bundle $L(D)=L_1\otimes\cdots\otimes L_k$ is very ample, see \cite{Fulton:gnus}. 
\end{proof}

A vector bundle $E$ which satisfies hypothesis of corollary~\ref{c2.2} is called {\it very ample}. For instance, the bundle
$E=\mathcal{O}_{\pp^1\times\pp^1\times\pp^1}(1,0,0)\oplus \mathcal{O}_{\pp^1\times\pp^1\times\pp^1}(0,1,1)$ is very ample on $X=\pp^1\times\pp^1\times\pp^1$.

We now use theorem~\ref{t2.1} and proposition~\ref{p2.1} to extend the classical notion of $(n-k)$-concavity in the projective space to an intrinsic notion of $E$-concavity in the smooth, complete toric variety $X$.

\subsection{$E$-concavity}\label{sec:2.3}

We call the parameter space for $E$-subvarieties the product of projective spaces   
$$
X^{'}=X^{'}(L_1,\ldots  ,L_k):=\pp(\Gamma(X,L_1))\times\cdots \times\pp(\Gamma(X,L_k)),
$$ 
equipped with the multi-homogeneous coordinates $a=(a_1,\ldots  ,a_k)$, where 
$$
a_i=(a_{im})_{m\in P_i\cap \zp^n}\in \pp(\Gamma(X,L_i))
$$
are the natural homogeneous coordinates for divisors in $|L_i|$. Thus any $a$ in $X^{'}$ determines the $E$-subvariety $C_a$ whose restriction to the torus $\mathbb{T}\subset X$ has Laurent polynomial equations 
$$
C_a\cap \mathbb{T}=\{l_1(a_1,t)=\cdots = l_k(a_k,t)=0\}.
$$
where $l_i(a_i,t)=\sum_{m\in P_i\cap \zp^n} a_{im}t^m$, for $i=1,\ldots,k$.

\begin{defi}
An open set $U$ of $X$ is called $E$-concave if any point of $U$ belongs to an $E$-subvariety included in $U$. The $E$-dual space of an $E$-concave set $U$ is the subset 
$$
U^{'}:=\{a\in X^{'}, C_a\subset U\}\subset X^{'}.
$$
The $E$-incidence variety over $U$ is the analytic subset of $U\times U^{'}$
$$ 
I_U:=\{(x,a)\in U \times U^{'}, x\in C_a\}
$$
equipped with its natural projections $p_U$ and $q_U$ on respectively $U$ and $U^{'}$. We note ${\rm Reg}(U^{'})$ the set of regular points $a\in U^{'}$, for which the subvariety $C_a$ is a smooth complete intersection.
\end{defi}

Certainly, $E$-concave open sets need not exist without any restrictive hypothesis on the algebraic vector bundle $E$. However, we know from section~\ref{sec:2.2} that ${\rm Reg}(X^{'})$ is an open Zariski subset in $X^{'}$ if $E$ is globally generated and essential. Next proposition shows that this condition is necessary and sufficient to be in the general situation described in \cite{Fabre:gnus} to generalize the Abel-transform. 

\begin{prop}\label{p2.3} A. These three assertions are equivalents:  
\begin{enumerate}
\item $E$ is globally generated; 
\item $I_X=I_X(E)$ is a bundle on $X$ in a product of projective spaces $\pp^{l_1-2}\times\cdots \times\pp^{l_k-2}$, where $l_i={\rm Card} (P_i\cap \zp^n)$, $i=1,\ldots,k$;
\item $I_X=I_X(E)$ is a smooth irreducible complete intersection in $X\times X^{'}$ and the morphism $p_X:  I_X\rightarrow X$ is a holomorphic submersion. 
\end{enumerate}
B. If $E$ is globally generated, then ${\rm Reg}\, (X^{'})$ is a non empty Zariski open set of $X^{'}$ if and only if $E$ is essential. In that case, the map $q_X : I_X\rightarrow X^{'}$ is holomorphic, proper, surjective and defines a submersion over ${\rm Reg}\, (X^{'})$. Moreover, the $E$-dual set $U^{'}$ of an $E$-concave open set $U\subset X$ is open, non empty, and connected if $U$ is.
\end{prop}

\begin{proof}
We show part A. For $i=1,\ldots,k$, let $V_i:=\Gamma(X,L_i)$ and $V_i^*$ its dual. We can define the tautological projective bundle ``points-hyperplanes'' over the projectivised space $\pp(V_i^*)$ of $V_i$
$$
T_i=\{(P,H)\in \pp(V_i)\times \pp(V_i^*)\,;\, P\in H\}
$$
where $\pp(V_i)\simeq \pp^{l_i-1}$. The product bundle $T=T_1\times\cdots \times T_k$ over $\pp(V_1^*)\times \cdots \times \pp(V_k^*)$ is a bundle in a product of projectives spaces isomorphic to  $\pp^{l_1-2}\times\cdots \times\pp^{l_k-2}$. If assertion $(1)$ is true, the morphism
\begin{eqnarray*}
\Phi=(\Phi_{L_1},\ldots  ,\Phi_{L_k}): X &\longrightarrow & \pp(V_1^*)\times \cdots \times \pp(V_k^*)
\end{eqnarray*} 
is well-defined, where, for $i=1\ldots, k$, $\Phi_{L_i}$ is the Kodaira map which sends $x$ to the point in $\pp(V_i^*)$ representing the hyperplane in $V_i$ of sections of $L_i$ vanishing at $x$. The triple $(I_X,X,p_X)$ is the pull-back bundle $\Phi^*(T)$ on $X$ which shows $(1)\Rightarrow (2)$. A projective bundle on a smooth irreducible variety $X$ is smooth irreducible and the projection on $X$ is a submersion. The variety $I_X$ is locally given by $k$ affine equations and is thus a complete intersection for dimension reasons which shows $(2)\Rightarrow (3)$. 
If a line bundle $L_i$ is not globally generated, the fiber $p_X^{-1}(x)$ over any $x$ in the base locus $BS(L_i)\ne\emptyset$ has codimension strictly less than $k$ in $\{x\}\times X^{'}$, contradicting $(3)$. So $(3)\Rightarrow (1)$, which completes the proof of $A$.

Let us show part B. The map $q_X$ is holomorphic, proper (since $X$ is compact) and surjective by construction. Proposition~\ref{p2.1} implies that ${\rm Reg}\,(X^{'})$ is a non empty Zariski open set of $X^{'}$ if and only if $E$ is essential. In that case, the fiber $q_X^{-1}(a)=C_a\times \{a\}$ over $a\in {\rm Reg}\,(X^{'})$ is smooth and the implicit function theorem implies that in a neighborhood of $(x,a)\in I_X$, the triple $(I_X,q_X,X^{'})$ is diffeomorphic to the triple $(W_x\times U_a,q,U_a)$ where $W_x$ is an open subset of $\cp^{n-k}$, $U_{a}$ is an open neighborood of $a$ and $q:W_x\times U_{a}\to U_a$ is the second projection. 

It remains to show that $U^{'}$ is open in $X^{'}$. Since $p_X:I_X \to X$ is a submersion, $I_U$ is open in  $I_X$ if and only if for all $x$ in $U$, the fiber $p_U^{-1}(x)=p_X^{-1}(x)\cap I_U$ is open in $p_X^{-1}(x)$. Let $F$ be the closed subset $X\setminus U$. We define for $x\in U$ the set
$$
W_x=\{(z,a)\in F\times q_X(p_X^{-1}(x)), z\in C_a\}\subset p_X^{-1}(x)\subset I_X. 
$$
The set $q_X(p_X^{-1}(x))$, isomorphic to a product of projectives spaces, is closed in $X^{'}$;   
the condition $z\in C_a$ is closed and the set $W_x$ is therefore closed in
$p_X^{-1}(x)\subset I_X$;  each fiber $p_U^{-1}(x)=p_X^{-1}(x)\setminus W_x$ 
is hence open in $p_X^{-1}(x)$ and $I_U=p_U^{-1}(U)$ is open in $p_X^{-1}(U)$, so in $I_X$. Since $q_X$ is holomorphic, it is open and the set $U^{'}=q_X(I_U)$ is open in $X^{'}$. If $U$ is connected, so is $I_U$ (because of its bundle structure), and so is $U^{'}= q_U(I_U)$ since $q_U$ is continuous.
\end{proof}

\subsection{$E$-duality}\label{sec:2.4}

From now on, we assume that $E=L_1\oplus\cdots\oplus L_k$ is an essential and globally generated vector bundle over $X$. 

\subsubsection{$E$-dual sets}\label{sec:2.4.1}

To any analytic subset $V$ in an $E$-concave open set $U$ of $X$, we associate (set theoretically) its $E$-dual set
$$
V^{'}:=q_U(p_U^{-1}(V))=\{a\in U^{'}, C_a\cap V \ne \emptyset\}.
$$
We define the incidence set over $V$
$$
I_V:=p_U^{-1}(V)=\{(x,a)\in V \times U^{'}, x \in C_a\}
$$
and note $p_V$ and $q_V$ the natural projections onto $V$ and $V^{'}$. 

\begin{prop}\label{p2.4} If $U$ is open $E$-concave, then
\begin{enumerate} 
\item The dual $V^{'}$ of a closed analytic set $V\subset U$ is closed analytic 
in $U^{'}$. Moreover, for any $\alpha \in V^{'}$, we have 
$$
codim_{\alpha} (V^{'})= k- r+ min\{ dim(V \cap C_a), a \in V^{'}\,  near\, \alpha\},
$$
where $r$ is the maximum of the dimensions of the irreducible components of $V$ meeting $C_{\alpha}$.
\item If $V$ is irreducible and if there exists $a\in {\rm Reg}\, (U^{'})$ such that the intersection 
$V\cap C_a$ is proper, its dual $V^{'}$ is irreducible of pure codimension
$$
codim (V^{'})= \begin{cases} \displaystyle k- dim (V), \quad  if\  dim (V) < k \\
0 \ otherwise.
\end{cases}
$$    
\end{enumerate} 
\end{prop}

\begin{proof}
$1$. Since $p_U$ is a submersion, $I_V$ is analytic closed in $I_U$ of codimension $n-{\rm dim}(V)$, irreducible if and only if $V$ is. The projection $q_U: I_U \rightarrow U^{'}$ is holomorphic, proper and the proper mapping theorem \cite{GR:gnus} implies that $V^{'}=q_U(I_V)$ is analytic closed in $U^{'}$, irreducible if $I_V$ is. Moreover, for any $\alpha \in V^{'}$:  
\begin{eqnarray*} 
{\rm dim}_{\alpha} (V^{'})=\max \{{\rm dim}_{(x,a)} (q_V),  
(x,a)\in I_V, a \in V^{'}\, {\rm near} \,\alpha\}
\end{eqnarray*}
where ${\rm dim}_{(x,a)}(q_V):={\rm dim}_{(x,a)} (I_V) - {\rm dim}_{(x,a)} (q_V^{-1}(a))$ is the codimension at 
$(x,a)$ in $I_V$ of the fiber $q_V^{-1}(a)={V\cap C_a}\times \{a\}$. On the other hand, since $p_U:I_U\to U$ is a submersion, we have ${\rm codim}_{I_U,(x,a)}(I_V)={\rm codim}_{U, x} (V)$  for any $(x,a)$ in $I_V$,  so that
\begin{eqnarray*} 
{\rm dim}_{(x,a)}(q_V) &=& {\rm dim}_{(x,a)} (I_U)-{\rm codim}_{U, x} (V)-{\rm dim}_x (V\cap C_a)\\
&=& n+{\rm dim} (X^{'}) - k-(n-{\rm dim}_x(V))-{\rm dim}_x (V\cap C_a)\\
&=& {\rm dim} (X^{'}) - (k-{\rm dim}_x (V)+{\rm dim}_x (V\cap C_a)).
\end{eqnarray*}
The maximum of the dimensions ${\rm dim}_x (V)$ for $x\in V\cap C_a$ when $a$ runs an arbitrary small open neighborhood $U_{\alpha}\subset U$ of $\alpha$ is the maximum of the dimensions of the irreducible components of $V$ meeting $C_{\alpha}$, which only depends of $\alpha$. Consequently, if $U_{\alpha}^*$ is small enough, the chosen definition of $r$ implies the equality
$$
{\rm dim}_{\alpha} (V^{'})= {\rm dim} (X^{'}) -(k-r+ \min \{{\rm dim} (V\cap C_a), a \in U_{\alpha}^*\cap V^{'}\} ), 
$$
which ends the proof of the first point.  
\vskip 1mm
\noindent
$2$. If $C_{a}$ and $V$ intersect properly for $a\in {\rm Reg}(U^{'})$, then
$$
{\rm dim} (V\cap C_{a}) = \begin{cases} {\rm dim} ( V )+ {\rm dim} (C_{a})-n \ {\rm if} \ {\rm dim} (V) \geq k \\
 0 \ {\rm if}\ {\rm dim} (V) < k\,.
\end{cases}  
$$
From the first point, we then have
$$
{\rm codim}_{a}\, (V^{'}) = \begin{cases} k-{\rm dim}  (V)+{\rm dim} (V) + {\rm dim} (C_{a})-n = 0 \ {\rm if}\ {\rm dim} (V) \geq k \\
k- {\rm dim} (V) \ {\rm if}\ {\rm dim} (V) < k \,.
\end{cases}
$$
Since $V$ is assumed irreducible, $V^{'}$ is irreducible, and ${\rm codim} (V^{'})={\rm codim}_{a}  (V^{'})$ which ends the proof.
\end{proof}

We remark that $E$-dual sets have a particular structure: if $V$ is closed analytic in an $E$-concave open set $U\subset X$, its dual $V^{'}=\cup_{x\in V} \, q_U(p_U^{-1}(x))$ is a union of products of projectives hyperplanes $\pp^{l_1-2}\times\cdots \times \pp^{l_k-2}\subset X^{'}$ restricted to $U^{'}$.

The following example illustrates proposition~\ref{p2.4}:
\begin{exem}
Let $X=\pp^1\times \pp^1\times\pp^1$, with natural multi-homogeneous coordinates $[x_0:x_1],[y_0:y_1],[z_0:z_1]$. Consider the essential globally generated bundle $E=\mathcal{O}_X(2,0,0)$. Then $X^{'}=X^{'}(E)$ is isomorphic to $\pp^2$ equipped with its natural homogeneous coordinates $[a_0:a_1:a_2]$, with the representation $C_a=\{p \in X\,;\, a_0x_0^2+a_1x_1^2+a_2 x_0x_1=0\}$. For $V=\{x_0=0\}$, the intersection $C_a\cap V$ is empty if $a_1\ne 0$ 
and $C_a\cap V=\{0\}\times \pp^1\times\pp^1$ otherwise. Thus, $V^{'}=\{a_1=0\}$ and the minimum of the dimension of $V\cap C_a$ for $a\in V^{'}$ near any points $[a_0:0:a_2]\in V^{'}$ is two and ${\rm codim}_{[a_0:0:a_2]} (V^{'})= 1-2+ 2=1$ as predicted by proposition~\ref{p2.4}.
\end{exem}

In this example, while ${\rm dim} (V)+{\rm dim}(C_a) = 4 > {\rm dim}(X)=3$, the intersection $V\cap C_a$ is generically empty (which traduces the inequality ${\rm codim}_{X^{'}} (V^{'}) >0$). This is an example of a set for which the intersection with an $E$-subscheme is never proper, situation excluded in the projective case $X=\pp^n$. Let us study those degenerate sets.

\subsubsection{Degenerate analytic sets}\label{sec:2.4.2}

Let $U$ be an $E$-concave open subset of $X$. 

\begin{defi} 
An analytic subset $V$ in $U$ is called $E$-degenerate if it contains an irreducible branch $V_0$ such that
$$
\begin{cases} 
{\rm dim}(V_0^{'})  <  {\rm dim} (I_{V_0})  \ {\rm if \ dim}( V_0)  \le k \\
{\rm dim} (V_0^{'})  <  {\rm dim} (U^{'})    \ {\rm if \ dim}( V_0)  >  k. \\
\end{cases}
$$
\end{defi}

\begin{prop}\label{p2.5}
Suppose that $V\subset U$ is an irreducible analytic subset of dimension $r\le k$ in $U$. 
The following assertions are equivalent:  
\begin{enumerate}
\item The set $V$ is not $E$-degenerate;  
\item The equality $codim (V^{'}) = k -  dim (V)$ holds; 
\item The analytic subset $\Upsilon_{V^{'}}:=\{a\in V^{'}, dim (V\cap C_a) \ge 1\}$ 
is of codimension at least two in $V^{'}$. 
\item There exists $a$ in $ Reg (U^{'})$ such that $ dim (V \cap C_a) =  0$.
\end{enumerate}  
\end{prop}

\begin{proof} 
$1\Rightarrow 2$. The equality ${\rm dim} (I_V)={\rm dim} (V)+ {\rm dim} (U^{'})-k$ and the inequality ${\rm dim} (I_V) \ge {\rm dim} (V^{'})$ ensure that $V$ is not $E$-degenerate if and only if ${\rm dim} (I_V)={\rm dim}( V^{'})$, that is ${\rm codim} (V^{'}) = k - {\rm dim}( V)$.
\vskip 1mm
\noindent
$2\Rightarrow 3$. If ${\rm dim} (\Upsilon_{V^{'}}( \ge {\rm dim} (V^{'})-1$, then ${\rm dim} (q_V^{-1}(\Upsilon_{V^{'}})) \ge 1+{\rm dim} (V^{'})-1$. If $V$ is not $E$-degenerate, ${\rm dim} (V^{'})={\rm dim} (I_V)$ so that $q_V^{-1}(\Upsilon_{V^{'}})=I_V$ and $V^{'}=\Upsilon_{V^{'}}$ by a dimension argument since $I_V$ and $V'$ are irreducible. Thus ${\rm dim}(V \cap C_a)\ge 1$ for any $a$ in $U^{'}$, contradicting $2$ by assertion $(1)$ of proposition~\ref{p2.4}.
\vskip 1mm
\noindent
$3\Rightarrow 4$. For all $x\in X$, a generic $E$-subscheme containing $x$ is a smooth complete intersection. Thus, the open subset $q_U(p_U^{-1}(V))\cap {\rm Reg}(U^{'})$ is dense in $V^{'}$ and meets $V^{'}\setminus \Upsilon_{V^{'}}$ under hypothesis $3$. 
\vskip 1mm
\noindent
$4\Rightarrow 1$ is a consequence of the second point in proposition~\ref{p2.4}. 
\end{proof}

For $r\ge k$, we obtain a similar caracterisation of $E$-degenerate set using the set  $\Upsilon_{V^{'}}:=\{a\in V^{'}, {\rm dim} (V\cap C_a) > r-k\}$. In particular, $X$ is not $E$-degenerate and the set $\Upsilon_{X^{'}}\subset X^{'}$ has codimension at least two. 

If $\rm{dim} (V)=k$, we note ${\rm Reg}_V(U^{'})$ the set of parameters $a\in U^{'}$ for which $C_a$ intersects $V$ transversaly outside its singular locus ${\rm Sing} (V)$.

\begin{coro}\label{c2.3}
Let $V\subset U$ be irreducible of dimension $k$. The following assertions are equivalent:  
\begin{enumerate}
\item The set $V$ is not $E$-degenerate;  
\item The set $Reg_V(U^{'})$ is nonempty;  
\item The set $Reg_V(U^{'})$ is dense in $U^{'}$.
\end{enumerate}
\end{coro}

\begin{proof} $(3)\Rightarrow (2)$ is trivial. If $a\in {\rm Reg}_V(U^{'})$, then ${\rm dim}(V\cap C_a)=0$ and $(2)\Rightarrow (1)$  follows from proposition~\ref{p2.5}. Let us show $(1)\Rightarrow (3)$.
Since ${\rm dim}({\rm Sing}(V)) < k$, then ${\rm dim}({\rm Sing} (V))^{'}) < {\rm dim} U^{'}$ from proposition~\ref{p2.4}. But $V^{'}$ is irreducible of codimension $0$ in the connected open set $U^{'}$, so $U^{'}=V^{'}$. For $a$ in $U^{'}$ generic, the intersection $V\cap C_a$ is finite ($X$ is compact) and does not meet ${\rm Sing}(V)$. As in proposition ~\ref{p2.1}, Sard's theorem implies that transversality $T_x X=T_x C_a\oplus T_x V$ is then generically satisfied. 
\end{proof}

We can interpret corollary~\ref{c2.3} in the algebraic case $U=X$. 

\begin{coro}\label{c2.4}
An irreducible subvariety $V$ of $X$ of dimension $k$ is $E$-degenerate if and only if its class $[V]\in A_k(X)$ is orthogonal to $L_1\cdots L_k$, that is if the intersection number $[V]\cdot L_1\cdots L_k$ is zero. In particular, there are no $E$-degenerate algebraic subvarieties if and only if $E$ is very ample.
\end{coro} 

\begin{proof}
This is immediate from proposition~\ref{p2.5} and corollaries~\ref{c2.2} and \ref{c2.3}.
\end{proof}

Thus, strict inequality $[V]\cdot L_1\cdots L_k>0$ is equivalent to that for any $E$-concave open set $U$ in $X$, there exists $a\in U^{'}$ for which the intersection $V\cap C_a$ is transversal, consisting in $[V]\cdot L_1\cdots L_k$ distinct points in $U$. In particular, {\it any $k$-dimensional closed subvariety which is not $E$-degenerate meets any $E$-concave open set}. 

\begin{coro}\label{c2.5} If $V\subset U$ is not $E$-degenerate and has pure dimension $r\le k$, the morphism $q_V : I_V \longrightarrow V^{'}$ is a ramified covering over the open subset  $V^{'}\setminus\Upsilon_{V^{'}}$ of degree $N=[\cp(I_V):\cp(V^{'})]$. 
\end{coro}

\begin{proof}
Clearly, the morphism $q_{V}$ restricts to a finite ramified covering of degree $N=[\cp(q_V^{-1} (V^{'} \setminus \Upsilon_{V^{'}})):\cp(V^{'} \setminus \Upsilon_{V^{'}})]$ over $V^{'} \setminus
\Upsilon_{V^{'}}$. By Proposition 5, the codimension of $\Upsilon_{V^{'}}$ in $V^{'}$ is at least two and $\cp(V^{'})=\cp(V^{'} \setminus\Upsilon_{V^{'}})$ by Hartog's extension theorem. Since $q_U$ is a proper submersion over the dense open set ${\rm Reg} (U^{'})\subset U^{'}$, the codimension in $q_U^{-1}(V^{'}\cap {\rm Reg} (U^{'}))$ of the analytic subset $q_U^{-1}(\Upsilon_{V^{'}}\cap {\rm Reg} (U^{'}))$ is at least two, and so is its analytic closure $q_U^{-1}(\Upsilon_{V^{'}})$ in  $q_U^{-1}(V^{'})=q_V^{-1}(V^{'})=I_V$. Consequently $\cp(I_V)=\cp(q_V^{-1} (V^{'} \setminus \Upsilon_{V^{'}}))$, which ends the proof. 
\end{proof}

If ${\rm dim}(V)<k$, a generic $E$-subscheme does not meet $V$ and one is tempted to think that the intersection $V\cap C_a$ consists in one point for a generic $a$ in $V^{'}$ (meaning that the subvarieties $I_V$ and $V^{'}$ are bimeromorphically equivalent). This is generally not true, as the following simple example shows~:

\begin{exem} 
Suppose $X=\pp^1\times\pp^1\times\pp^1$ and let $E$ be the essential globally generated bundle $\mathcal{O}_X((2,0,0),(0,1,0))$. A generic $E$-subscheme $C_a$ is the disjoint union $C_a=(\{P_1\}\times\{P\}\times\pp^1) \cup (\{P_2\}\times\{P\}\times\pp^1)$
where the points $P_1$ and $P_2$ are distinct and belong to the first factor $\pp^1$ while the point $P$ belongs to the second factor $\pp^1$. For $V=\pp^1\times\{[0:1]\}\times\{[0:1]\}$ the set $C_a\cap V$ is generically empty. However, for $a$ in $V^{'}$ generic, the intersection $V\cap C_a$ consists of two distincts points.
\end{exem}

To understand the behaviour of $E$-duality with rational equivalence, we need to consider the case of cycles, taking now account of multiplicities. 

\subsubsection{$E$-duality for cycles}\label{sec:2.4.3}

Let $U\subset X$ be an open $E$-concave subset and let $V$ be an irreducible analytic subset of $U$. 
We define the $E$-dual cycle $V^*$ of $V$ as follow. If $V$ is $E$-degenerate or ${\rm dim}(V) > k$, then $V^*:=0$. Otherwise, the field $\cp(I_V)$ is a finite extension on $\cp(V^{'})$ and we set
$$
V^*=[\cp(I_V): \cp(V^{'})]\cdot V^{'}.
$$
We extend by linearity this definition to the case of cycles. $E$-duality agrees with rational equivalence:

\begin{prop}\label{p2.6} A. Let $l= dim(X^{'})$. The map $V\mapsto V^*$ induces a morphism of graded $\zp$-free modules on the Chow groups of $X$ and $X^{'}$
$$
A_{j}(X)\rightarrow A_{l-k+j}(X^{'})
$$
for all $j=0,\ldots  ,k$.

B. Suppose that the line bundles $L_1,\ldots,L_k$ are very ample and let $W$ be an effective $(k-1)$-cycle on $X$ of class $[W]=\sum_{\tau\in\Sigma(n-k+1)} \nu_{\tau} [V(\tau)]$ in $A_{n-k+1}(X)$. 
Then $W^*$ is an effective divisor in the product of projective spaces $X^{'}$ of multidegree $(d_1,\cdots,d_k)\in \zp^k$, where
$$
d_i=\sum_{\tau\in\Sigma(n-k+1)} \nu_{\tau} MV_{k-1}(P_1^{\tau},\cdots,\hat{P_i^{\tau}},\cdots,P_k^{\tau})
$$
for all $i=1,\ldots,k$ (we omit the $i$-th polytope).
\end{prop}

\begin{proof}
A. We have ${\rm dim}(V^*)={\rm dim}(V)+l-k$ by proposition~\ref{p2.4}. The map $V\mapsto V^*$ is the composed map of the usual pull-back map of cycles under the submersion $p_U$ with the direct total image of cycles under the proper morphism $q_U$. Thus, it is compatible with rational equivalence (Appendix A of \cite{Hartshorne:gnus}). 

B. Let $R_{W}$ be the multihomogeneous equation of the effective divisor $W^*$ in the product of projective spaces $X^{'}$, with the convention that $R_{W_0}=1$ for irreducible $E$-degenerate components $W_0$ of $W$. We call this polynomial the {\it $E$-resultant of $W$}. For $\tau\in\Sigma(n-k+1)$, the $E$-resultant $R_{V(\tau)}$ corresponds to the classical $(L_1^{\tau},\cdots,L_k^{\tau})$-resultant of the $k$ very ample line bundles $L_1^{\tau},\cdots,L_k^{\tau}$ on the $(k-1)$-dimensional toric variety $V(\tau)$, as defined in \cite{GKZ:gnus}. If $W$ is rationally equivalent to the cycle $\sum_{\tau\in\Sigma(n-k+1)} \nu_{\tau} V(\tau)$, part A implies by linearity that 
$$
{\rm deg}_{a_i} R_{W}=\sum_{\tau\in\Sigma(n-k+1)} \nu_{\tau} {\rm deg}_{a_i} R^{\tau}_{E}.
$$
From \cite{GKZ:gnus}, the partial degree in $a_i$ of the polynomial $R_{V(\tau)}$ is equal to the intersection number of $V(\tau)$ with a generic $E_i$-subvariety where $E_i:=\oplus_{j=1,j\ne i}^k L_j$. We conclude with proposition~\ref{p2.2}. 
\end{proof}

\section{The toric Abel-transform}\label{sec:3}

\subsection{Meromorphic forms and residue currents}\label{sec:3.1}

Let $X$ be an analytic manifold of dimension $n$. Any $k$-dimensional analytic subset $V \subset X$ gives rise to a positive $d$-closed $(k,k)$-current $[V]$ on $U$,  acting by integration on the regular part of $V$. 
We define the sheaf of $\mathcal{O}_X$-modules $\mathcal{M}^q_V$ on $X$ of germs of {\it meromorphic $q$-forms} on $V$ as the restriction to $V$ of the sheaf of germs of meromorphic $q$-forms in the ambient space $X$ whose polar sets intersect properly $V$. We note $M^q_V$ the corresponding vector space of global sections. Note that the restriction map $M^q_X\longrightarrow M^q_V$ is in general not surjective.

As shown in \cite{hp:gnus}, any $q$-meromorphic form $\phi\in M^q_V$ gives rise to a current $[V]\land \phi$ on $U$, supported on $V$ and acting on any $(k-q,k)$-test-forms $\Psi$ using the {\it principal value criterion}:  
\begin{displaymath}
\langle [V]\land\phi , \Psi\rangle := \lim_{\ep\to 0}\int_{V \cap \{|g|>\ep\}}\phi\land\Psi
\end{displaymath}
where $g$ is any holomorphic function on $U$ not identically zero on $V$ but vanishing on the singular locus of $V$ and on the polar set of $\phi$ (the limit does not depend on the choice of $g$). The polar locus of $\phi$ is the analytic subset 
$$
{\rm pol}(\phi)=\{p\in X,\,\,\bar{\partial}([V]\land \phi)\ne 0\}\subset V.
$$
By Hartog's theorem, ${\rm pol}(\phi)$ is either a codimension $1$ analytic subset of $V$, either the empty set. 

A meromorphic form $\phi$ is called {\it abelian} or {\it regular} at $x\in V$ if the current $[V]\land \phi$ is $\bar{\partial}$-closed at $x$. We note $\om^q_V$ the corresponding sheaf of $\mathcal{O}_X$-modules. When $V$ is smooth, $\om^q_V$ is the usual sheaf $\Omega^q_V$ of germs of holomorphic $q$-forms on the manifold $V$. 

\subsection{The Abel-transform}\label{sec:3.2}

Let $E=L_1\oplus\cdots\oplus L_k$ be a globally generated essential rank $k$ split vector bundle on a smooth complete toric variety $X=X_{\Sigma}$. We keep the notations of section~\ref{sec:2}.
Let $U\subset X$ be an open $E$-concave set. We suppose $U$ connected for simplicity. 

For any $k$-dimensional analytic subset $V\subset U$ and any $q$-form $\phi\in M^q_V$, the {\it Toric Abel-transform} (relatively to $E$) of the locally residual current $[V]\land \phi$ is the current on $U^{'}$
$$
\mca([V]\land\phi):= q_{U*}(p_{U}^*([V]\land\phi)),
$$
where $q_{U*}$ is the push-forward map for currents associated to the proper morphism $q_U$ and
$$
p_{U}^*( [V]\land\phi):=[I_V]\land p_{U}^*\phi.
$$
This current is well defined since $p_U$ is a holomorphic submersion by proposition~\ref{p2.5}, so that $p_{U}^*\phi$ is meromorphic on $I_V$.

If $V$ is not $E$-degenerate, the intersection $V\cap C_{a}$ is generically transversal by corollary~\ref{c2.3} and consists in $N$ points $\{p_1(a),\ldots,p_N(a)\}$ outside the polar locus of $\phi$, whose coordinates vary holomorphically with $a$ by the implicit function theorem. For such generic $a$, the Abel-transform $\mca([V]\land\phi)$ coincides with the holomorphic $q$-form
$$
Tr_V \phi (a)=\sum_{i=1}^N p_i^* \phi.
$$
We call it the {\it Trace of $\phi$ on $V$} (relatively to $E$). 

Let us introduce residue calculus in the picture.

\subsection{Residual representation of the Abel-transform}\label{sec:3.3}
 
For simplicity, we suppose that $V\subset U$ is an analytic hypersurface meeting properly the codimension one orbits of $X$. The more general case of a locally complete intersection can be treated in the same way. So $E=L_1\oplus\cdots\oplus L_{n-1}$ is an essential globally generated bundle of rank $n-1$, where the line bundles  $L_i$ are attached to Cartier divisors $D_i$, with polytopes $P_i$, for $i=1,\ldots,n-1$.

Using the principle of unicity of analytic continuation, we can compute the trace in a sufficiently small open set of $U^{'}$ (always noted $U^{'}$) such that for every $a\in U^{'}$, the intersection $V\cap C_a$ is contained in the torus (this is possible by proposition~\ref{p2.5}). We can thus use the torus variables $t=(t_1,\ldots,t_n)$. We recall that 
$$
C_a\cap \mathbb{T}=\{l_1(a_1,t)=\cdots=l_{n-1}(a_{n-1},t)=0\}
$$
where the Laurent polynomials $l_i(a_i,\cdot)$ are supported by the polytopes $P_i$, for $i=1,\ldots,n-1$.

\begin{prop}\label{p3.1}
The coefficients of the meromorphic $q$-form $Tr_V\phi\in M^q(U^{'})$ are Grothendieck residues of meromorphic $n$-forms (in $t$) depending meromorphically of the parameters $a\in U^{'}$. If $f$ is the analytic equation of $V$ near $V\cap C_a$, $a\in U^{'}$, then
\begin{eqnarray*}
Tr_V\phi=\sum_{|I|=q}\sum_{M\in \prod_{i\in I}P_i}{\rm Res}\, \left[\begin{matrix}
t^{|M|}\phi \land df \land\bigwedge\limits_{j\notin I} dl_j \cr 
f(t), l_1(a_1,t),\ldots  ,l_{n-1}(a_{n-1},t)
\end{matrix}
\right]\,da_M
\end{eqnarray*}
where $M=(m_{i_1},\ldots  ,m_{i_{q}})$, $|M|=m_{i_1}+\cdots+m_{i_{q}}$,  $da_M=\land_{i\in I} da_{i m_i}$ and 
$$
{\rm Res}\, \left[\begin{matrix}
t^{m}\phi \land df \cr 
f, l_1,\ldots  ,l_{n-1}
\end{matrix}
\right]:=\sum_{p\in U}{\rm res}_p \Big(\frac{t^{m}\phi \land df}{f\cdot l_1(a_1,\cdot)\cdots l_{n-1}(a_{n-1},\cdot)}\Big)
$$
denotes the action in $U$ of the Grothendieck residues defined by the polynomials $(f,l_1,\ldots,l_{n-1})$ on the meromorphic $n$-form $t^{m}\phi \land df$ (the residues are zero except eventually on the finite set $C_a\cap V$).
\end{prop}

\begin{proof}
The meromorphic form $Tr_V\phi=q_{U*}p_{U}^*([V]\land\phi)$ is a $(q,0)$-current on $U^{'}$ which acts on test-forms $\varphi$ by
\begin{eqnarray*}
\langle Tr_V\, (\phi)\,,\,\varphi\rangle &=&
\int_{U^{'}}  Tr_V (\phi) (a)\land \varphi(a) \\
&=&\int_{ I_U} ([p_U^{-1}(V) (a)]\wedge
p_{U}^* [\Phi]) \land {q_U}^* (\varphi)(t,a)\\
&=&\int_{U\times U^{'}} ([V](t)\land \phi(t)\land[I_U](t,a))
\land \varphi(a).
\end{eqnarray*}
For any complete intersection $Z=\{g_1=\cdots=g_k=0\}$ in a complex manifold, the $(k,k)$-current of integration $[Z]$ is attached to the $(k,0)$-residue current 
\begin{displaymath}
\langle \bar{\partial}\frac{1}{g_1}\land\cdots\land\bar{\partial}\frac{1}{g_k},\Psi \rangle:=\big(\frac{1}{2i\pi}\big)^k \lim_{\underline{\ep}\rightarrow 0}\int_{|g_1|=\ep_1,\ldots|g_k|=\ep_k}\frac{\Psi}{g_1\cdots g_k}
\end{displaymath}
by the Poincar\'{e}-Lelong equation 
$$
[Z]=dg_1\land\cdots\land dg_k \land \bar{\partial}\frac{1}{g_1}\land\cdots\land\bar{\partial}\frac{1}{g_k}.
$$
In our situation, this formula gives a local expression for the $(2n-1,n)$-current $T:=[V]\land[I_U]\wedge \phi$ on $U\times U^{'}$
$$
T=\phi\land df\wedge \Big(
\bigwedge\limits_{i=1}^{n-1} d_{(t,a)}l_i\Big)
\wedge \overline\partial \Big[\frac{1}{f}\Big]\wedge
\Big(\bigwedge\limits_{i=1}^{n-1}
\bar{\partial}_{(t,a)} \Big[\frac{1}{l_i}\Big]\Big).
$$
Since the current $Tr_V\phi=q_{U*} T$ acts on test-forms of bidegree $(l-q,l)$ \textit{in $a$}, where $l={\rm dim}(X^{'})$, we have $Tr_V\phi=q_{U*} T'$, where
$$
T'=\sum_{I\subset\{1,\ldots,n-1\}, |I|=q}\phi\land df\wedge \Big(
\bigwedge_{i\notin I} d_{t}l_i\Big)\wedge \Big(
\bigwedge_{j\in I} d_{a_j}l_j\Big)
\wedge \overline\partial \Big[\frac{1}{f}\Big]\wedge
\Big(\bigwedge\limits_{i=1}^{n-1}
\bar{\partial}_{t} \Big[\frac{1}{l_i}\Big]\Big).
$$
Since 
$$
\langle \overline\partial \Big[\frac{1}{f}\Big]\wedge
\Big(\bigwedge\limits_{i=1}^{n-1}
\bar{\partial}_{(t,a)} \Big[\frac{1}{l_i}\Big]\Big)\,,\, \psi\rangle={\rm Res}\, \left[\begin{matrix}
\psi \cr 
f, l_1,\ldots ,l_{n-1}
\end{matrix}
\right]
$$
for any meromorphic $n$-form $\psi$, this gives the desired residual expression for the trace using linearity of Grothendieck residues and expanding the form $\bigwedge\limits_{j\in I} d_{a_i}l_i$ . 
\end{proof}

\begin{rema}
In the algebraic case $U=X$, the previous residual representation allows to compute explicitly the polar divisor for the toric Abel-transform using results in \cite{GK:gnus} and \cite{DK:gnus} who give denominators formulae for toric residues depending on parameters (see \cite{WM3:gnus} for the very ample case).
\end{rema}

We are now in position to prove the Abel-inverse theorem in smooth complete toric varieties.

\section{A toric version of the Abel-inverse theorem}\label{sec:3.4}

Let $E=L_1\oplus\cdots\oplus L_{n-1}$ be an essential globally generated vector bundle on $X$ with polytopes $P_1,\ldots,P_{n-1}$ defined as before. We suppose that $E$ satisfies the following additional condition~:
\vskip 2mm
\noindent
\textit{There exists a chart $U_{\sigma}$, $\sigma\in \Sigma(n)$ such that each polytope $\Delta_{i,\sigma}$ $i=1,\ldots,n-1$ defined in section~\ref{sec:2.1} contains the elementary simplex of $\rp^n$.  $\hfill (*)$}
\vskip 2mm
\noindent
Condition $(*)$ means that ${\rm dim}\, H^0(V(\tau),L_{i|V(\tau)})\ge 2$, $i=1,\ldots,n-1$ for any  $n-1$-dimensional cone $\tau\subset\sigma$. There always exist bundles $E$ who satisy $(*)$. Geometrically, it means that a generic hypersurface $H\in|L_i|$ meets any one dimensional orbit closure $V(\tau)$ meeting $U_{\sigma}$. Let $U_0:=U_{\sigma}$ be such a chart. We prove the following toric version of the classical Abel-inverse theorem:

\begin{theo}\label{t3.2}
Let $V$ be an analytic hypersurface of an $E$-concave open set $U\subset X$ with no components included in $X\setminus U_0$, and let $\phi$ be a meromorphic $(n-1)$-form on $V$ which does not vanish identically on any component of $V$. Then, there exists an algebraic hypersurface $\wt{V}\subset X$ and a rational form $\wt{\phi}$ on $\wt{V}$ such that 
$$
\wt{V}_{|U}=V\,\,\,\,\, and\,\,\,\,\wt{\phi}_{|V}=\phi
$$
if and only if the meromorphic form $Tr_V\phi$ is rational in the constant coefficients $a_0=(a_{1,0},\ldots,a_{n-1,0})$ of the $n-1$ polynomial equations of $C_a$ in $U_0$.
\end{theo}

\begin{rema}
The condition $\wt{V}_{|U}=V$ is equivalent to that the intersection number $\wt{V}\cdot L_1\cdots L_{n-1}$ is equal to the cardinality $N$ of the finite set $V\cap C_a$ for a generic $a\in U^{'}$. From proposition~\ref{p2.2}, this is equivalent to that the intersection of the hypersurface $\wt{V}$ with the torus $\mathbb{T}$ is given by a Laurent polynomial whose Newton polytope $P$ satisfies $MV(P,P_1,\ldots,P_{n-1})=N$. For the complete characterization of the class of $\wt{V}$ in the Picard group ${\rm Pic}(X)$ (or, equivalently the polytope $P$ of $\wt{V}$), we need supplementary degree conditions in terms of traces of appropriate rational functions on $X$ (see \cite{WM3:gnus}).
\end{rema}

\begin{proof} 
\textit{Direct implication.} Since $\wt{V}_{|U}=V$, then $V\cap\, C_a = \wt{V}\cap\, C_a$ for any $a\in U^{'}$ and the form $Tr_V \phi$ coincides on $U^{'}$ with the Abel-transform $\mathcal{A}([\wt{V}]\land\wt{\phi})$. This  $(q,0)$-current is defined on the product of projective spaces $X^{'}$, and $\bar{\partial}$-closed outside an hypersurface dual to the polar locus of $\wt{\phi}$ on $\wt{V}$. This current thus corresponds to a meromorphic form on $X^{'}$, which, by the GAGA principle, is rational in $a$ (so in $a_0$).
\vskip1mm
\noindent
\textit{Converse implication.} Under a monomial change of coordinates on the torus, we can suppose that the affine coordinates $x=(x_1,\ldots,x_n):=(x_{1,\sigma},\ldots,x_{n,\sigma})$ of $U_0=U_{\sigma}$ coincide with the torus coordinates $t=(t_1,\ldots,t_n)$ so that the polytopes $\Delta_{D_i,\sigma}$ coincide with the polytopes $P_i$ attached to the line bundle $L_i$. Thus, every $E$-subvariety $C_a$ has affine equations
$$
C_a\cap U_{0}=\{l_1(a_1,x)=\cdots=l_{n-1}(a_{n-1},x)=0\}
$$
where the polynomials $l_i(a_i,\cdot)$ are supported by the polytopes $P_i\subset (\rp^{+})^n$ containing the elementary simplex of $\rp^n$ with a constant term  $a_{i0}$ for $i=1,\ldots,n-1$.

By assumption, the sets $V\cap (X\setminus U_0)$ and $V\cap {\rm pol}(\phi)$ have codimension at least $2$ in $U$. Thus, we can suppose that $U^{'}$ is a neighborhood of a point $\alpha\in {\rm Reg}_V(U^{'})$ such that the intersection 
$$
V\cap C_a=\{p_1(a),\ldots,p_N(a)\}
$$
is transversal, does not meet the polar locus of $\phi$ and is contained in $U_0$ for all $a\in U^{'}$. So we suppose $V$ smooth, included in $U_0$, and $\phi$ holomorphic on $V$. In particular, we can use affine coordinates $(x_1,\ldots,x_n)$ of the chart $U_0$ to compute traces.

We now extend to the toric situation a lemma of ``propagation'', crucial in the original proof in \cite{hp:gnus}. 

\subsection{The propagation principle}\label{sec:3.4.2}
We show here the following lemma.

\begin{lemm}\label{l3.1}
If $Tr_V\phi$ is rational in $a_{i0}$, then so is the form $Tr_V h\phi$ for every polynomial $h\in \cp[x_1,\ldots,x_n]$ supported in the convex cone generated by $P_i$.
\end{lemm}

\begin{proof}
We note $f=0$ the equation of $V$ in $U$ (so that $f$ is holomorphic near the finite set $\{p_1(\alpha),\ldots,p_N(\alpha)\}$) and we define 
$$
v_m:={\rm Res}\, \left[\begin{matrix}
x^m \phi\land df \cr 
f, l_1,\ldots  ,l_{n-1}
\end{matrix}
\right]
$$
for every $m\in \zp$. Proposition~\ref{p3.1} applied  with $q=n-1$ implies the equality
\begin{eqnarray*}
Tr_V\, x^m\phi=\sum_{M\in P_1\times\cdots\times P_{n-1}}{\rm Res}\, \left[\begin{matrix}
x^{|M|+m}\phi \land df  \cr 
f(x), l_1(a_1,x),\ldots  ,l_{n-1}(a_{n-1},x)
\end{matrix}
\right]\,da_M
\end{eqnarray*}
for all $m\in\zp^n$. Thus we need to show that the meromorphic functions $v_{m}$ on $U^{'}$ are rational in $a_{i0}$ for all $m$ in $P+kP_i$ and all $k\in \np$. For $k=0$, this is our hypothesis since the meromorphic functions $v_m$ are coefficients of the trace $Tr_V\phi$. We suppose $v_{m'}$ rational in $a_{i0}$ for $m'\in P+kP_i$ and we show that $v_{m+m'}$ is rational in $a_{i0}$ for every $m\in P_i$.

Using Cauchy integral representation for Grothendieck residues and Stokes theorem, we see that for all $m\in P_i$ and all $m'\in \np^n$, we have
\begin{eqnarray*}
\partial_{a_{im}}v_{m'} & = & -{\rm Res}\, \left[\begin{matrix}
x^{m'} \partial_{a_{im}}l_i\phi\land df \cr 
f, l_1,\ldots,l_i^2,\ldots  ,l_{n-1}
\end{matrix}
\right]\\
& = & -{\rm Res}\, \left[\begin{matrix}
x^{m+m'} \phi\land df \cr 
f, l_1,\ldots,l_i^2,\ldots  ,l_{n-1}
\end{matrix}
\right] = \partial_{a_{i0}}v_{m+m'}.
\end{eqnarray*}
This is also equivalent to that the form $Tr_V x^m\phi$ is closed on $U^{'}$, since the form $x^m\phi$ is of maximal degree on $V$ and the $d$-operator commutes with $p_{U}^*$ and $q_{U*}$. 

So, if $m'\in kP_i$, our hypothesis implies that $\partial_{a_{i0}}v_{m+m'}$ is rational in $a_{i0}$, and we want to show that it has no simple pole in $a_{i0}$ in its decomposition into simple elements in $a_{i0}$. If $\partial_{a_{i0}}[v_{m+m'}]=0$, it's trivially true. Otherwise, there exists $c\in\cp^*$ such that the functions $v_{m'}$ and $v_{m'}+c\, v_{m+m'}$ are $\cp$-linearly independent. Then:  
$$
\partial_{a_{i0}}[v_{m'}+c v_{m+m'}]=\partial_{a_{i0}}[v_{m'}]+c\, 
\partial_{a_{i m}}[v_{m'}]=\partial_{a_{i0}+ca_{i m}}[v_{m'}]
$$
so that the function $\partial_{a_{i0}}[v_{m'}+c v_{m+m'}]$ (rational in $a_{i0}$) admits the two linearly independents primitives $v_{m'}+c v_{m+m'}$ and $v_{m'}$ in the two linearly independent directions
$a_{i0}$ and $a_{i0}+c\, a_{i m}$. Thus it can not have simple poles in its decomposition in simple elements in $a_{i0}$ and the function $v_{m'}+c v_{m+m'}$ is rational in $a_{i0}$. Since $v_{m'}$ is assumed to be rational in $a_{i0}$, so is $v_{m+m'}$. 
\end{proof}

Let us come back to the proof of theorem~\ref{t3.2}.

\subsection{The inversion process}

We devide it in two steps.
\vskip 1mm
\noindent
{\it Step 1. Extension of $V$.} The finite degree-$N$ field extension $[\cp(I_V):\cp(U^{'})]$ is generated by the meromorphic functions $x_1,\ldots,x_n$ considered as elements of $\cp(I_V)$. Thus, for $c=(c_1,\ldots,c_n)\in \cp^n\setminus\{0\}$ generic, the meromorphic function defined by $y=c_1x_1+\cdots+c_nx_n$ on $I_V$ is a primitive element for this extension. We note $y_j(a)$ the analytic functions $y(p_j(a))$ for $j=1,\ldots,N$. The unitary polynomial of $\cp(U^{'})[Y]$ 
\begin{eqnarray*}
P(a,Y) & :=& (Y-y_1(a))\times\cdots\times (Y-y_N(a))\\
&=& Y^N+\sigma_{N-1}(a)Y^{N-1}+\ldots+\sigma_0(a),
\end{eqnarray*}
has degree $N$ and satisfies $P(a,y)_{I_V}\equiv 0$. It is thus the minimal (and the characteristic) polynomial of $y$. Let us define the meromorphic function on $U^{'}$ 
$$
w_k:={\rm Res}\, \left[\begin{matrix}
y^k \phi\land df \cr 
f, l_1,\ldots,l_i,\ldots  ,l_{n-1}
\end{matrix}
\right],
$$
and let us show that the matrix of traces
$$
M:=\left(\begin{matrix}
w_0 & \ldots   &  w_{N-1} \\
w_1 & \ldots   &  w_{N} \\
\vdots & \ddots & \vdots \\
w_{N-1} & \ldots   &  w_{2N-2} 
\end{matrix}
\right)
$$
is invertible in $\cp(U^{'})$. Since $P(a,y)\in \cp(U\times U^{'})$ vanishes identically on the reduced set $I_V$, then 
$$
{\rm Res}\, \left[\begin{matrix}
y^k P(a,y)\phi\land df \cr 
f, l_1,\ldots,l_i,\ldots  ,l_{n-1}
\end{matrix}
\right]=0
$$
for any $k\in \np$. Expanding this condition by linearity for $k=0,\ldots,N-1$, we obtain
$$
M \left(\begin{matrix} \sigma_0 \\
\vdots \\
\sigma_{N-1}\end{matrix}
\right)=\left(\begin{matrix} -w_N \\
\vdots \\
-w_{2N-1}
\end{matrix}
\right).
$$
Similary, any vector $\lambda=(\lambda_0,\ldots,\lambda_{N-1})\in(\cp(U^{'}))^N$ such that $M \lambda^t = 0$ represents the coefficients of a polynomial $F=F(X,a)\in \cp(U^{'})[X]$ of degree $N-1$ such that
\begin{eqnarray}
{\rm Res}\, \left[\begin{matrix}
F(y,a) y^k \phi \land df \cr 
f(x), l_1(a_1,x),\ldots  ,l_{n-1}(a_{n-1},x)
\end{matrix}
\right]=0,\quad k=0,\ldots N-1.
\end{eqnarray}
Let $m\in \np^n$. Since $y$ defines a primitive element for the degree $N$ extension $[\cp(I_V):\cp(U^{'})]$, there exists a polynomial $R(a,Y)\in \cp(U^{'})[X]$ of degree at most $N-1$ such that $x^m_{|I_V}\equiv R(a,y)_{|I_V}$. By linearity, this implies that the relation $(2)$ still holds when replacing the polynomial $y^{k}$ by any monomial $x^m$. Thus, if $\phi\land df = h(x)dx_1\land\cdots\land dx_n$, then 
$$
F(y,a) h(x) \in (f(x),l_1(a_1,x),\ldots,l_{n-1}(a_{n-1},x))
$$
by the duality theorem. By assumption on $\phi$, the holomorphic function $g$ is not identically zero on any components of $V=\{f=0\}$, so that $F(y,a)$ must vanish identically on the incidence variety $I_V$. By a degree argument, this forces $F\equiv 0$, so that $M$ is invertible in $\cp(U^{'})$.

Since the line bundles $L_i$ satisfy $(*)$, we have $\rp^+ P_i = (\rp^+)^n$ for every $i=1,\ldots,n-1$ and lemma~\ref{l3.1} implies that the meromorphic functions $w_k$ are rational in $a_0$ for every $k\in \np$. Since $M$ is invertible, the coefficients of $P$ are thus rational in $a_0$. Trivially,
$$
x\in C_a\cap U_0 \Longleftrightarrow a_{i0}=a_{i0}-l_i(a_i,x)=:l'_i(a'_i,x),\,\,\forall i=1,\ldots, n-1
$$
where $a_i=(a_{i0},a'_i)$. The function 
\begin{eqnarray*}
Q_{c,a'}(x) &:=& P\big(l'_1(a'_1,x),a'_1,\ldots,l'_{n-1}(a'_{n-1},x),a'_{n-1},c_1x_1+\cdots+c_nx_n\big)\\
&=& \prod_{j=1}^N \Big(y-y_j\big(l'_1(a'_1,x),a'_1,\ldots,l'_{n-1}(a'_{n-1},x),a'_{n-1}\big)\Big)
\end{eqnarray*}
is thus rational in $x$ and defines an algebraic hypersurface
$$
V_{c,a'}:=\{x\in U_0,\,\,\,Q_{c,a'}(x)=0\}
$$
which contains $V$ for every $a'$ in a neighborhood of $\alpha'$. For generic $c\in \cp^n$ the sum $y=\sum_1^n c_ix_i$ remains a primitive element for the extension $[\cp(I_V):\cp(U^{'})]$, and previous construction gives a set
$$
V_0:=\bigcap_{c\,{\rm generic}, a'\,{\rm near}\,\alpha'} V_{c,a'}
$$
which is algebraic and contains $V$. By construction, if $q\in V_0\cap C_a$ for $a\in U^{'}$, there exists $j\in \{1,\ldots,N\}$ such that 
$$
c_1x_1(q)+\cdots+c_nx_n(q)=c_1x_1(p_j(a))+\cdots+c_nx_n(p_j(a))
$$
for $c\in \cp^n$ generic. This implies that $q \in\{ p_1(a),\ldots,p_N(a)\}$ so that 
$$
V_0\cap C_a = V\cap C_a\quad {\rm for}\, {\rm all}\,a\in U^{'}.
$$
Let $\wt{V}$ be the Zariski closure of $V_0$ to $X$. If $C_{a}$ meets $\wt{V}$ outside $V$ for $a\in U^{'}$, then it would remain true in a neighborhood of $a$, contradicting proposition~\ref{p2.4} since $\wt{V}$ meets properly the hypersurface at infinity $X \setminus U_0$. Thus
$$
\wt{V}\cap C_a= V\cap C_a\quad {\rm for \,\, all}\,\,a\in U^{'}.
$$
This shows that $\wt{V}_{|U}=V\cup V'$ where $V'\cap C_a=\emptyset$ for all $a\in U^{'}$. Since $U$ is $E$-concave, this forces $V'$ to be empty and $\wt{V}_{|U}=V$.
\vskip 2mm
\noindent
{\it Step 2. Extension of $\phi$.} By assumption on $U^{'}$, $\phi$ is holomorphic on the smooth analytic set $V$ and can be identified with a holomorphic form in the ambient space. Thus, there exists a holomorphic function $h$ in a neighborhood of the finite set $\{p_1,\ldots,p_N\}$  such that
$$
\phi\land df = h(x)dx_1\land\cdots\land dx_n.
$$
Let $y=c_1x_1+\cdots+c_nx_n$ be as before and consider the Lagrange interpolation polynomial
\begin{eqnarray*}
H(a,Y)&:=&\sum_{j=1}^N \prod_{r=1,r\ne j}^N \frac{Y-y_r(a)}{y_j(a)-y_r(a)}h(p_j(a))\\
&=& \tau_{N-1}(a) Y^{N-1}+\ldots+\tau_1(a) Y + \tau_0 (a).
\end{eqnarray*}
The polynomial $H(a,Y)\in \cp(U^{'})[Y]$ satisfies $H(a,y_j(a))=h(p_j(a))$ for all $j=1,\ldots,N$ and all $a\in U^{'}$, that is
\begin{eqnarray}
H(y,a)_{|I_{V}}=p_{U}^*(h)_{|I_{V}}.
\end{eqnarray}
Thus we have equality
\begin{eqnarray*}
&{\rm Res}\, \left[\begin{matrix}
y^k H(a,y) dx_1\land\cdots\land dx_n \cr 
f(x), l_1(a_1,x),\ldots  ,l_{n-1}(a_{n-1},x)
\end{matrix}
\right]&\\
=&{\rm Res}\, \left[\begin{matrix}
y^k \phi\land df \cr 
f(x), l_1(a_1,x),\ldots  ,l_{n-1}(a_{n-1},x)
\end{matrix}
\right]&
\end{eqnarray*}
for all $k\in \np$. That means that the $N$-uple $(\tau_0,\ldots  ,\tau_{N-1})$ satisfies
\begin{eqnarray}
\begin{matrix}
t_{N-1}\tau_{N-1} &+& \cdots &+& t_0\tau_{0} = w_0 \\
&\vdots &\ddots &\vdots & \vdots \\
t_{2N-2}\tau_{N-1} &+&\cdots &+& t_{N-1}\tau_{0} = w_{N-1},
\end{matrix}
\end{eqnarray}
where 
$$
t_k:={\rm Res}\, \left[\begin{matrix}
y^k dx_1\land\cdots\land dx_n \cr 
f(x), l_1(a_1,x),\ldots  ,l_{n-1}(a_{n-1},x)
\end{matrix}
\right].
$$
As before, the duality theorem implies that system $(3)$ is Cramer. From Point 1, $f$ can be replaced by a polynomial $\wt{f}\in \cp[x_1,\ldots,x_n]$ for which $V_0=\{\wt{f}=0\}$. Thus, the functions $t_k$ are rational in $a$ (so in $a_0$) for every $k\in\np$ while the functions $w_k$ are rational by hypothesis. So $H(a,Y)$ depends rationally on $a_0$ and for any $a=(a_0,a')\in U^{'}$, the function 
\begin{eqnarray*} 
g(a',x) :=  H\big(l'_1(a'_1,x),a'_1,\ldots,l'_{n-1}(a'_{n-1},x),a'_{n-1},c_1x_1+\cdots+c_nx_n\big)
\end{eqnarray*} 
is rational in $x$. From $(2)$ it satisfies $g(a',x)_{|I_{V}}=p_{U}^*(h)_{|I_V}$, so that the rational function $\wt{h}_{a'}(x):=g(a',x)$ coincides with $h$ on $V$ independently of $a'$. The inner product of the rational $n$-form $\wt{h}_{a'}dx_1\land\cdots \land dx_{n}$ with $d\wt{f}$ defines a rational form $\wt{\phi}_{a'}$ on $X$ which satisfies $\wt{\phi}_{a'|V}=\phi$. If the polar locus of $\wt{\phi}_{a'}$ does not meet properly the algebraic hypersurface $\wt{V}$, this also holds in the $E$-concave set $U$, contradicting that $\phi\in M^{n-1}_V$. Thus $\wt{\phi}_{a'}$ defines a rational form $\wt{\phi}$ on $\wt{V}$ which is equal to $\phi$ on $V$. Theorem~\ref{t3.2} is proved. 
\end{proof}

\backmatter

\end{document}